\newfont{\Bbb}{msbm10 scaled\magstephalf}
\documentclass[reqno]{amsart}
\usepackage{amssymb}
\usepackage{amsfonts}
\usepackage{amsmath, amsthm, amscd, amsfonts}
\usepackage{cases}
\usepackage{collref}
\usepackage{cite}

\newtheorem{Lemma}{Lemma}
\newtheorem{Theorem}{Theorem}

\begin{document}
\title[ Unitary equivalence and reducing subspaces of analytic Toeplitz operator]{On unitary equivalence and reducing subspaces of analytic Toeplitz operator on vector-valued Hardy space}

\author[C. Chen, Yucheng Li and Y. Wang] {Cui Chen, Yucheng Li$^*$ and Ya Wang}
\address{\newline Cui Chen\newline Department of Mathematics
\newline Tianjin University of Finance and Economics\newline Tianjin 300222\newline P.R. China}
\email{chencui\_cc@126.com}

\address{\newline  Yucheng Li\newline Department of Mathematics\newline Hebei Normal University\newline Shijiazhuang 050024\newline P.R. China.}
\email{liyucheng@hebtu.edu.cn}

\address{\newline Ya Wang\newline Department of Mathematics
\newline Tianjin University of Finance and Economics\newline Tianjin 300222\newline P.R. China}
\email{wangyasjxsy0802@163.com}




\date{}
\thanks{\noindent $^*$Corresponding author.\\
This work was supported by the National Natural Science Foundation of
China (Grant No. 12201452, 12171138).}

\begin{abstract}
In this paper, we proved that $T_{z^n}$ acting on the $\mathbb{C}^m$-valued Hardy space $H_{\mathbb{C}^m}^2(\mathbb{D})$, is unitarily equivalent to $\bigoplus_1^{mn}T_z$, where $T_z$ is acting on the scalar-valued Hardy space $H_{\mathbb{C}}^2(\mathbb{D})$. And using the matrix manipulations combined with operator theory methods, we completely describe the reducing subspaces of $T_{z^n}$ on $H_{\mathbb{C}^m}^2(\mathbb{D})$.
\end{abstract}

\maketitle

\section{Introduction}

Let $\mathbb{D}$ be the open unit disk in the complex plane $\mathbb{C}$. The classical scalar-valued Hardy space $H_\mathbb{C}^2(\mathbb{D})$ consists of all analytic functions $f$ that satisfy
$$\|f\|=\Big[\sup_{0\leq r<1}\frac{1}{2\pi}
\int_0^{2\pi}|f(re^{i\theta})|^2d\theta\Big]^{\frac{1}{2}}<\infty.$$
Correspondingly, the $\mathbb{C}^m$-valued Hardy space \cite{PJR}, denoted by $H_{\mathbb{C}^m}^2(\mathbb{D})$, can be defined as the collection of all $\mathbb{C}^m$-valued analytic functions $F$ on $\mathbb{D}$ such that
$$\|F\|=\Big[\sup_{0\leq r<1}\frac{1}{2\pi}
\int_0^{2\pi}|F(re^{i\theta})|^2d\theta\Big]^{\frac{1}{2}}<\infty.$$
It can also be defined by
$$H_{\mathbb{C}^m}^2(\mathbb{D}):=\{F(z)=\sum_{k=0}^\infty A_k z^k:\|F\|^2=\sum_{k=0}^\infty \|A_k\|_{\mathbb{C}^m}^2<\infty, A_k\in\mathbb{C}^m\}.$$
$H_{\mathbb{C}^m}^2(\mathbb{D})$ is an analytic Hilbert space with respect to the inner product $\langle F,G\rangle=\sum_{i=1}^m\langle f_i,g_i\rangle_{H_\mathbb{C}^2(\mathbb{D})}$
for functions $F=(f_1,f_2,\cdots,f_m)^{\rm T},G=(g_1,g_2,\cdots,g_m)^{\rm T}\in H_{\mathbb{C}^m}^2(\mathbb{D})$. We can also view $H_{\mathbb{C}^m}^2(\mathbb{D})$ as the direct sum of $m$-copies of $H_\mathbb{C}^2(\mathbb{D})$ or sometimes it is useful to see the above space as a tensor product of two Hilbert spaces $H_\mathbb{C}^2(\mathbb{D})$ and $\mathbb{C}^m$, that is,
$$H_{\mathbb{C}^m}^2(\mathbb{D})\equiv \mathop{\underbrace{H_\mathbb{C}^2(\mathbb{D})\oplus\cdots\oplus H_\mathbb{C}^2(\mathbb{D})}}\limits_{m}\equiv H_\mathbb{C}^2(\mathbb{D})\otimes\mathbb{C}^m.$$
Let $T_{z}$ denote the forward shift operator (multiplication by the independent variable) acting on $H_{\mathbb{C}^m}^2(\mathbb{D})$, that is, $T_{z}F(z)=zF(z), z\in\mathbb{D}$. The Banach space of all $\mathcal{L}(\mathbb{C}^m)$ (set of all bounded linear operators on $\mathbb{C}^m$)-valued bounded analytic functions on $\mathbb{D}$ is denoted by $H_{\mathcal{L}(\mathbb{C}^m)}^\infty(\mathbb{D})$. Each $\Theta\in H_{\mathcal{L}(\mathbb{C}^m)}^\infty(\mathbb{D})$ induces an analytic Toeplitz operators $T_\Theta\in H_{\mathcal{L}(\mathbb{C}^m)}^\infty(\mathbb{D})$ defined by
$$T_\Theta F(z)=\Theta(z)F(z),\;\;\;\;F\in H_{\mathbb{C}^m}^2(\mathbb{D}).$$
The elements of $H_{\mathcal{L}(\mathbb{C}^m)}^\infty(\mathbb{D})$ are called the multipliers and are determined by
$$\Theta\in H_{\mathcal{L}(\mathbb{C}^m)}^\infty(\mathbb{D})\;\mbox{if and only if}\;T_{z}T_\Theta=T_\Theta T_{z}.$$
See \cite{CDP} for more details about $H_{\mathbb{C}^m}^2(\mathbb{D})$. It is clear that $T_{z^n}$, that defined by $T_{z^n}F(z)=z^nF(z)$, belongs to $H_{\mathcal{L}(\mathbb{C}^m)}^\infty(\mathbb{D})$. For easier understanding, we rewrite the definition of analytic Toeplitz $T_{z}$ and $T_{z^n}$ as following:
\begin{eqnarray*}
T_{z}F(z)=\left[
\begin{array}{ccc}z&&\\&\ddots&\\&&z\\\end{array}\right]_{\tiny{m\times m}}
\left[\begin{array}{c}f_1\\ \vdots\\f_m\\\end{array}\right],
T_{z^n}F(z)=\left[
\begin{array}{ccc}z^n&&\\&\ddots&\\&&z^n\\\end{array}\right]_{m\times m}
\left[\begin{array}{c}f_1\\\vdots\\f_m\\\end{array}\right]
\end{eqnarray*}
for function $F=(f_1,f_2,\cdots,f_m)^{\rm T}\in H_{\mathbb{C}^m}^2(\mathbb{D})$.

Given two bounded linear operators $S, T$ on a Hilbert space $H$, we say that $S$ is {\em similar} to $T$ if $A^{-1}SA=T$ for some bounded linear invertible operator $A$ on $H$, in which case we write $S\sim T$. If $A$ is  unitary, i.e., $A^*A=AA^*=I$,  in this case, we call $S$ is unitarily equivalent to $T$.
Also, a closed subspace $M$ of a Hilbert space $H$ is called an {\em invariant} subspace for a bounded linear operator $T$ if $TM\subset M$. If $M$ and $M^{\bot}$ are both invariant subspaces for $T$, then $M$ is said to be a {\em reducing} subspace for $T$. In addition, if $M$ does not contain any nontrivial reducing subspace, then $M$ is called a {\em minimal} reducing subspace for $T$.

It is well known that invariant subspace of an operator can be identified in terms of invariant subspaces of a unitary equivalent operator or a similar operator. Over the years, there are some results characterizing the lattice of closed invariant subspaces by constructing unitary equivalent operators or similar operators, we refer the interested readers to the recent papers such as \cite{CR,GZ,Rs}.

In previous years, the unitary equivalence and similarity between analytic Toeplitz operators $T_{z^n}$ and $\oplus_1^n T_z$ acting on a scalar-valued Hilbert space is an active topic which has be concerned in lots of papers. In those papers, sometimes the analytic Toeplitz operators $T_{z^n}$ and $T_z$ were known as multiplication operators, therefore, they were denoted by $M_{z^n}$ and $M_z$. In the earlier paper, Ball studied the reducing subspace problem in Hardy space (see \cite{Ball1, Ball2}). Inspired by it, Jiang and Li (see \cite{JL}) in 2007 first obtained that analytic Toeplitz operator $M_{B(z)}$ is similar to $\oplus_1^n M_z$ on the Bergman space if and only if $B(z)$ is an $n$-Blaschke product. Next, Li (see \cite{L1}) in 2009 proved that multiplication operator $M_{z^n}$ is similar to $\oplus_1^n M_z$ on the weighted Bergman space. And then, Jiang and Zheng \cite{JZ} extended the main result in \cite{JL} to the weighted Bergman space. In 2011, Douglas and Kim in \cite{DK} investigated the reducing subspaces for an analytic multiplication operator $M_{z^n}$ on the Bergman space $A_\alpha^2(A_r)$ of the annulus $A_r$. Guo and Huang studied the similarity, unitary equivalence and reducing subspaces of multiplication operators on the Bergman space (see \cite{GH})).  Moreover, in 2017, the unitary equivalence of analytic multipliers on Sobolev disk algebra was discussed in \cite{CQW}. But such a characterization is not always hold for all Hilbert spaces. In recent paper, Li, Lan and Liu (see \cite{LLL}) proved that multiplication operator $M_{z^n}$ is not similar to $\oplus_1^n M_z$ on the Fock space, they are quasi-similar, actually. For further results related, see \cite{AH,JS,LM,ZRF}, which including some advances extending to higher dimensions.

It should be noted that all these results were considered in the scalar-valued Hilbert spaces, such a characterization leave a blank in the vector-valued Hilbert space. Based on the works above, in this paper we are interested in the corresponding result in the vector-valued Hilbert space. Specifically, we proved that $T_{z^n}$ acting on the $\mathbb{C}^m$-valued Hardy space $H_{\mathbb{C}^m}^2(\mathbb{D})$ is unitarily equivalent to $\bigoplus_1^{mn}T_z$, where $T_z$ is acting on the scalar-valued Hardy space $H_{\mathbb{C}}^2(\mathbb{D})$. And using the matrix manipulations combined with operator theory methods, we completely describe the reducing subspaces of $T_{z^n}$ on $H_{\mathbb{C}^m}^2(\mathbb{D})$.

\section{The unitary equivalence of analytic toeplitz operator}

For $i=1,2,\cdots,m$, let $\{e_i\}_{i=1}^m$ be the $m$-dimensional unit vector set with $e_i=(0,\cdots,0,\mathop{1}\limits_{i},0,\cdots,0)^{\rm T}$. Then we can get the orthonormal basis of $H_{\mathbb{C}^m}^2(\mathbb{D})$.
\begin{Lemma}
$\{e_iz^p=(0,\cdots,0,\mathop{z^p}\limits_{i},0,\cdots,0)^{\rm T}\}_{\substack{i=1,2,\cdots,m\\
p=0,1,2,\cdots}}$ forms an orthonormal basis of $H_{\mathbb{C}^m}^2(\mathbb{D})$.
\end{Lemma}

\begin{proof}
Choose any $F\in H_{\mathbb{C}^m}^2(\mathbb{D})$ with $F(z)=\sum\limits_{p=0}^\infty A_pz^p
=\sum\limits_{p=0}^\infty (a_{1p},a_{2p},\cdots,a_{mp})^{\rm T}z^p$, we can write it as
\begin{eqnarray*}
F(z)&=&\sum_{p=0}^\infty(a_{1p}e_1+a_{2p}e_2+\cdots+a_{mp}e_m)z^p\\
&=&\sum_{i=1}^m\sum_{p=0}^\infty a_{ip}e_iz^p.
\end{eqnarray*}
And it is obvious that $\{e_iz^p=(0,\cdots,0,\mathop{z^p}\limits_{i},0,\cdots,0)^{\rm T}\}_{\substack{i=1,2,\cdots,m\\
p=0,1,2,\cdots}}$ is linear independent. Thus, $\{e_iz^p\}_{\substack{i=1,2,\cdots,m\\p=0,1,2,\cdots}}$ forms a basis of $H_{\mathbb{C}^m}^2(\mathbb{D})$.

Moreover, note that $\langle e_iz^p,e_jz^q\rangle=1$ if and only if $i=j$ and $p=q$. Otherwise, $\langle e_iz^p,e_jz^q\rangle=0$. So $\{e_iz^p\}_{\substack{i=1,2,\cdots,m\\
p=0,1,2,\cdots}}$ forms an orthonormal basis of $H_{\mathbb{C}^m}^2(\mathbb{D})$.
\end{proof}

Before going to the main theorem, we need one more lemmas.

\begin{Lemma}
For each $i=1,2,\cdots,m$ and $j=0,1,\cdots,n-1$, set $\mathcal{L}_{ij}={\rm Span}\{e_iz^{nk+j}:k\in\mathbb{N}\}$. Then

(i) $\{e_iz^{nk+j}:k\in\mathbb{N}\}$ forms an orthonormal basis of $\mathcal{L}_{ij}$.

(ii) $H_{\mathbb{C}^m}^2(\mathbb{D})=\bigoplus_{ij}\mathcal{L}_{ij}$, where the direct sum is over all indexes $i=1,2,\cdots,m$ and $j=0,1,\cdots,n-1$.

(iii) $\mathcal{L}_{ij}$ is a reducing subspace for $T_{z^n}$.
\end{Lemma}
\begin{proof}
(i) By a simple calculation, we have
\begin{eqnarray*}
\langle e_iz^{nk_1+j},e_iz^{nk_2+j}\rangle=1
\end{eqnarray*}
if and only if $k_1=k_2$. Otherwise, $\langle e_iz^{nk_1+j},e_iz^{nk_2+j}\rangle=0$. Thus (i) holds.

(ii) It is clear that $\mathcal{L}_{ij}\perp\mathcal{L}_{pq}$ for all $(i,j)\neq(p,q)$. Next, for $F(z)\in H_{\mathbb{C}^m}^2(\mathbb{D})$ with $F(z)=\sum\limits_{p=0}^\infty A_pz^p
=\sum\limits_{p=0}^\infty (a_{1p},a_{2p},\cdots,a_{mp})^{\rm T}z^p$, it is easy to see that $F$ has the form
\begin{eqnarray*}
F(z)=\sum_{i=1}^m\sum_{p=0}^\infty a_{ip}e_iz^p=
\sum_{i=1}^m\sum_{j=0}^{n-1}\sum_{k=0}^\infty a_{i,nk+j}e_iz^{nk+j}.
\end{eqnarray*}
Suppose $F=0$, then from
\begin{eqnarray*}
\langle\sum_{i=1}^m\sum_{j=0}^{n-1}\sum_{k=0}^\infty a_{i,nk+j}e_iz^{nk+j},e_qz^l\rangle=0\;\;\;\mbox{for each}\;q\in\{1,2,\cdots,m\}\;\mbox{and}\;l\in\mathbb{N},
\end{eqnarray*}
we conclude that $a_{i,nk+j}=0$ for all $i=1,2,\cdots,m, j=0,1,\cdots,n-1$ and $k\in\mathbb{N}$. That is, $0=\overbrace{0\oplus0\oplus\cdots\oplus0}^{mn}$, which yields that $H_{\mathbb{C}^m}^2(\mathbb{D})=\bigoplus_{ij}\mathcal{L}_{ij}$.

(iii) It is obvious that both $\mathcal{L}_{ij}$ and $\mathcal{L}_{ij}^\perp$ are invariant subspaces for $T_{z^n}$.
\end{proof}

By the previous lemma, it is clear that $T_{z^n}=\bigoplus_{ij}T_{z^n}|_{\mathcal{L}_{ij}}$. Then we can get the unitary equivalence of analytic Toeplitz operator.

\begin{Theorem}
$T_{z^n}$ acting on the vector-valued Hardy space $H_{\mathbb{C}^m}^2(\mathbb{D})$ is unitarily equivalent to $\bigoplus_1^{mn}T_z$, where $T_z$ is acting on the scalar-valued Hardy space $H_{\mathbb{C}}^2(\mathbb{D})$.
\end{Theorem}
\begin{proof}
Note that $T_zz^k=z^{k+1}.$ Set $T_{ij}=T_{z^n}|_{\mathcal{L}_{ij}}$, then
\begin{eqnarray*}
T_{ij}e_iz^{nk+j}=e_iz^{n(k+1)+j}.
\end{eqnarray*}
Define $X_{ij}:H_{\mathbb{C}}^2(\mathbb{D})\rightarrow\mathcal{L}_{ij}$ such that $X_{ij}z^k=e_iz^{nk+j}$.
Thus we conclude that $X_{ij}T_z=T_{ij}X_{ij}$. In fact, for each $k\in\mathbb{N}$,
\begin{eqnarray*}
X_{ij} T_z z^k=X_{ij}z^{k+1}
=e_iz^{n(k+1)+j}=T_{ij}e_iz^{nk+j}=T_{ij}X_{ij}z^k.
\end{eqnarray*}
Obviously, $X_{ij}$ is bounded, invertible, and $\|X_{ij}\|=1$. So $T_{ij}$ is unitarily equivalent to $T_z$. Thus, we are done.
\end{proof}

\section{The reducing subspace of analytic toeplitz operator $T_{z^n}$}

As is well known, a reducing subspace for a bounded operator $T$ on a Hilbert space $H$ can be characterized by a
certain commuting property:
If $M$ is a closed subspace of $H$ and $P_M$ is a projection onto $M$, it turns out that $M$ is a reducing subspace for $T$ if and only if $P_{M}T=TP_{M}$.

For convenience, 
we use ${\rm I}$ and ${\rm O}$ to represent the identity matrix and the zero matrix, respectively.
\begin{Lemma}
Let $H_{\mathbb{C}}^2(\mathbb{D})$ be the scalar-valued Hardy space. If the operator $P$ has the matrix representation
\begin{equation*}
P=\left[
\begin{array}{cccccc}
p_{11}&p_{12}&p_{13}&\cdots&p_{1k}&\cdots\\
p_{21}&p_{22}&p_{23}&\cdots&p_{2k}&\cdots\\
p_{31}&p_{32}&p_{33}&\cdots&p_{3k}&\cdots\\
\vdots&\vdots&\vdots&\vdots&\vdots&\cdots\\
p_{k1}&p_{k2}&p_{k3}&\cdots&p_{kk}&\cdots\\
\vdots&\vdots&\vdots&\vdots&\vdots&\ddots
\end{array}
\right]
\end{equation*}
with respect to the orthonormal basis $\{z_k\}_{k\in\mathbb{N}}$ of $H_{\mathbb{C}}^2(\mathbb{D})$, where $p_{ij}\in\mathbb{C}, \forall i,j\geq 1$. Then $P\in\mathcal{A}'(T_z)$ if and only if the entries of $P$ satisfying the following equalities:
\begin{eqnarray*}
\left\{
\begin{array}{ll}
p_{ij}=0,& i<j,\\
p_{ii}=p_{11}, &i=2,3,\cdots,\\
p_{j+k,j}=p_{j+k-1,j-1}, &j\geq2,k\geq1.
\end{array}
\right.
\end{eqnarray*}
\end{Lemma}

\begin{proof}
Since $T_z z^k=z^{k+1}$, then
\begin{eqnarray*}
&&T_z(1,z,z^2,\cdots,z^k,\cdots)=(z,z^2,z^3\cdots,z^{k+1},\cdots)\\
&=&(1,z,z^2,\cdots,z^k,\cdots)
\left[
\begin{array}{ccccccc}
0&0&0&\cdots&0&0&\cdots\\
1&0&0&\cdots&0&0&\cdots\\
0&1&0&\cdots&0&0&\cdots\\
\vdots&\vdots&\vdots&\ddots&\vdots&\vdots&\cdots\\
0&0&0&\cdots&0&0&\cdots\\
0&0&0&\cdots&1&0&\cdots\\
\vdots&\vdots&\vdots&\vdots&\vdots&\vdots&\ddots
\end{array}
\right].
\end{eqnarray*}
If $P\in\mathcal{A}'(T_z)$, that is, $T_zP=PT_z$, then we can obtain that
\begin{eqnarray*}
\left[
\begin{array}{cccccc}
0&0&0&\cdots&0&\cdots\\
p_{11}&p_{12}&p_{13}&\cdots&p_{1k}&\cdots\\
p_{21}&p_{22}&p_{23}&\cdots&p_{2k}&\cdots\\
\vdots&\vdots&\vdots&\vdots&\vdots&\cdots\\
p_{k-1,1}&p_{k-1,2}&p_{k-1,3}&\cdots&p_{k-1,k}&\cdots\\
\vdots&\vdots&\vdots&\vdots&\vdots&\ddots
\end{array}
\right]
\end{eqnarray*}
\begin{eqnarray*}
=\left[
\begin{array}{cccccc}
p_{12}&p_{13}&p_{14}&\cdots&p_{1,k+1}&\cdots\\
p_{22}&p_{23}&p_{24}&\cdots&p_{2,k+1}&\cdots\\
p_{32}&p_{33}&p_{34}&\cdots&p_{3,k+1}&\cdots\\
\vdots&\vdots&\vdots&\vdots&\vdots&\cdots\\
p_{k2}&p_{k3}&p_{k4}&\cdots&p_{k,k+1}&\cdots\\
\vdots&\vdots&\vdots&\vdots&\vdots&\ddots
\end{array}
\right],
\end{eqnarray*}
which concludes that
\begin{eqnarray}
\left\{
\begin{array}{ll}
p_{ij}=0,& i<j,\\
p_{ii}=p_{11}, &i=2,3,\cdots,\\
p_{j+k,j}=p_{j+k-1,j-1}, &j\geq2,k\geq1.\label{1}
\end{array}
\right.
\end{eqnarray}
Conversely, if the entries of $P$ satisfy (\ref{1}), a.e., $P$ admits the following matrix representation
\begin{eqnarray*}
P=\left[
\begin{array}{cccccc}
p_{11}&0&0&\cdots&0&\cdots\\
p_{21}&p_{11}&0&\cdots&0&\cdots\\
p_{31}&p_{21}&p_{11}&\cdots&0&\cdots\\
\vdots&\vdots&\vdots&\vdots&\vdots&\cdots\\
p_{k1}&p_{k-1,1}&p_{k-2,1}&\cdots&p_{11}&\cdots\\
\vdots&\vdots&\vdots&\vdots&\vdots&\ddots
\end{array}
\right]
\end{eqnarray*}
with respect to the basis $\{z_k\}_{k\in\mathbb{N}}$. Simple computation shows that $T_zP=PT_z$, thus $P\in\mathcal{A}'(T_z)$.
\end{proof}

\begin{Lemma}
Let $H_{\mathbb{C}}^2(\mathbb{D})$ be the scalar-valued Hardy space. If the operator $G$ is a projection, then $G\in\mathcal{A}'(\bigoplus_1^{mn}T_z)$ if and only if $G$ admits the following matrix representation
\begin{scriptsize}\begin{equation*}
\left[
\begin{array}{ccccccccccccc}
G_{10}&&&&&&&&&&&&\\
&G_{20}&&&&&&&&&&&\\
&&\ddots&&&&&&&&&\\
&&&G_{m0}&&&&&&&&&\\
&&&&G_{11}&&&&&&&&\\
&&&&&G_{21}&&&&&&&\\
&&&&&&\ddots&&&&&&\\
&&&&&&&G_{m1}&&&&&\\
&&&&&&&&\ddots&&&&\\
&&&&&&&&&G_{1,n-1}&&&\\
&&&&&&&&&&G_{2,n-1}&&\\
&&&&&&&&&&&\ddots&\\
&&&&&&&&&&&&G_{m,n-1}\\
\end{array}
\right],
\end{equation*}\end{scriptsize} where $G_{ij}={\rm I}\;\mbox{or}\;{\rm O}$ for $i=1,2,\cdots,m,j=0,1,\cdots n-1$.
\end{Lemma}
\begin{proof}
By Lemma 3, we know that $P\in\mathcal{A}'(T_z)$ if and only if
\begin{eqnarray*}
P=\left[
\begin{array}{cccccc}
p_{11}&0&0&\cdots&0&\cdots\\
p_{21}&p_{11}&0&\cdots&0&\cdots\\
p_{31}&p_{21}&p_{11}&\cdots&0&\cdots\\
\vdots&\vdots&\vdots&\vdots&\vdots&\cdots\\
p_{k1}&p_{k-1,1}&p_{k-2,1}&\cdots&p_{11}&\cdots\\
\vdots&\vdots&\vdots&\vdots&\vdots&\ddots
\end{array}
\right].
\end{eqnarray*}
Moreover, if $P$ is a projection, then $P=P^*=P^2$, thus $p_{i1}=0$  for all $i\geq 2$, and $p_{11}=1\;\mbox{or}\;0$. 
That is $P={\rm I}\;\mbox{or}\;{\rm O}$. Therefore, if $G$ is a projection, and $G\in\mathcal{A}'(\bigoplus_1^{mn}T_z)$ if and only if $G$ admits the following matrix representation
\begin{scriptsize}
\begin{equation*}
\left[
\begin{array}{ccccccccccccc}
G_{10}&&&&&&&&&&&&\\
&G_{20}&&&&&&&&&&&\\
&&\ddots&&&&&&&&&\\
&&&G_{m0}&&&&&&&&&\\
&&&&G_{11}&&&&&&&&\\
&&&&&G_{21}&&&&&&&\\
&&&&&&\ddots&&&&&&\\
&&&&&&&G_{m1}&&&&&\\
&&&&&&&&\ddots&&&&\\
&&&&&&&&&G_{1,n-1}&&&\\
&&&&&&&&&&G_{2,n-1}&&\\
&&&&&&&&&&&\ddots&\\
&&&&&&&&&&&&G_{m,n-1}\\
\end{array}
\right],
\end{equation*}
\end{scriptsize}where $G_{ij}={\rm I}\;\mbox{or}\;{\rm O}$ for all $i=1,2,\cdots,m,\;j=0,1,\cdots,n-1$.
\end{proof}

\begin{Theorem}
$T_{z^n}$ has $2^{mn}$ reducing subspaces with $mn$ minimal reducing subspaces $\mathcal{L}_{ij}$ for $i=1,2,\cdots,m$ and $j=0,1,\cdots,n-1$.
\end{Theorem}
\begin{proof}
From the proof of Theorem 1, we know $T_z=X_{ij}^*T_{ij}X_{ij}$. Set
\begin{scriptsize}
\begin{eqnarray*}
X=\left[
\begin{array}{ccccccccccccc}
X_{10}&&&&&&&&&&&&\\
&X_{20}&&&&&&&&&&&\\
&&\ddots&&&&&&&&&\\
&&&X_{m0}&&&&&&&&&\\
&&&&X_{11}&&&&&&&&\\
&&&&&X_{21}&&&&&&&\\
&&&&&&\ddots&&&&&&\\
&&&&&&&X_{m1}&&&&&\\
&&&&&&&&\ddots&&&&\\
&&&&&&&&&X_{1,n-1}&&&\\
&&&&&&&&&&X_{2,n-1}&&\\
&&&&&&&&&&&\ddots&\\
&&&&&&&&&&&&X_{m,n-1}\\
\end{array}
\right],
\end{eqnarray*}
\end{scriptsize}then $X$ is a unitary operator satisfying $\bigoplus_1^{mn}T_z=X^*T_{z^n}X$. Note that $G$ is a projection and $G\in\mathcal{A}'(\bigoplus_1^{mn}T_z)$, $G$ has the form described as Lemma 4. Thus we have
\begin{eqnarray*}
G\bigoplus_1^{mn}T_z=\bigoplus_1^{mn}T_z G &\Rightarrow&GX^*T_{z^n}X=X^*T_{z^n}XG\\
&\Rightarrow& XGX^*T_{z^n}=T_{z^n}XGX^*\\
&\Rightarrow& GXX^*T_{z^n}=T_{z^n}GXX^*\\
&\Rightarrow& G T_{z^n}=T_{z^n}G,
\end{eqnarray*}
where
\begin{scriptsize}
\begin{equation*}
G=\left[
\begin{array}{ccccccccccccc}
G_{10}&&&&&&&&&&&&\\
&G_{20}&&&&&&&&&&&\\
&&\ddots&&&&&&&&&\\
&&&G_{m0}&&&&&&&&&\\
&&&&G_{11}&&&&&&&&\\
&&&&&G_{21}&&&&&&&\\
&&&&&&\ddots&&&&&&\\
&&&&&&&G_{m1}&&&&&\\
&&&&&&&&\ddots&&&&\\
&&&&&&&&&G_{1,n-1}&&&\\
&&&&&&&&&&G_{2,n-1}&&\\
&&&&&&&&&&&\ddots&\\
&&&&&&&&&&&&G_{m,n-1}\\
\end{array}
\right]
\end{equation*}
\end{scriptsize}and $G_{ij}={\rm I}\;\mbox{or}\;{\rm O}$ for all $i=1,2,\cdots,m,\,j=0,1,\cdots,n-1$. By Lemma 2, we concludes that the reducing subspaces of $T_{z^n}$ are
$$\bigoplus_{j=0}^{n-1}\bigoplus_{i=1}^m c_{ij}\mathcal{L}_{ij},\;\;\;\mbox{where}\;
c_{ij}=1\;\mbox{or}\;0,$$
and the minimal reducing subspaces are $\mathcal{L}_{ij},i=1,2,\cdots,m,j=0,1,\cdots,n-1$.
\end{proof}

\end{document}